\documentclass[11pt]{article}

\usepackage{latexsym,color,amsmath,amsthm,amssymb,amscd,amsfonts}
\usepackage{graphicx}
\usepackage{cancel}

\newtheorem{conj}{Conjecture}[section]
\newtheorem{theorem}[conj]{Theorem}

\newtheorem{remark}[conj]{Remark}
\newtheorem{lemma}[conj]{Lemma}

\newtheorem{corollary}[conj]{Corollary}

\newtheorem{cor}[conj]{Corollary}

\newcommand\independent{\protect\mathpalette{\protect\independent}{\perp}} 
\def\independent#1#2{\mathrel{\rlap{$#1#2$}\mkern2mu{#1#2}}}

\newcommand{\Conv}{\mathrm{Conv}}

\newcommand{\R}{\mathbb{R}}

\renewcommand{\P}{\mathbb{P}}

\newcommand{\E}{\mathbb{E}}

\newcommand{\Z}{\mathcal{Z}}

\renewcommand{\Z}{\mathbb{Z}}

\newcommand{\N}{\mathbb{N}}

\date{}

\author{Abdulmajeed Alqasem$^a$, Heshan Aravinda$^a$, Arnaud Marsiglietti$^{a,*}$, James Melbourne$^b$}

\title{On a Conjecture of Feige for Discrete Log-Concave Distributions}

\begin{document}

\maketitle

\vskip5mm
\noindent
$^a$ University of Florida, Department of Mathematics, P.O. Box 118105, Gainesville, FL 32611, USA \\
$^b$ Centro de Investigaci\'on en Matem\'aticas (CIMAT), Probabilidad y Estad\'istica, Guanajuato, Gto 36023, M\'exico \\
$^*$Corresponding author, a.marsiglietti@ufl.edu

\vskip1cm

\begin{abstract}

A remarkable conjecture of Feige (2006) asserts that for any collection of $n$ independent non-negative random variables $X_1, X_2, \dots, X_n$, each with expectation at most $1$, 
$$ \P(X < \E[X] + 1) \geq \frac{1}{e}, $$
where $X = \sum_{i=1}^n X_i$. In this paper, we investigate this conjecture for the class of discrete log-concave probability distributions and we prove a strengthened version. More specifically, we show that the conjectured bound $1/e$ holds when $X_i$'s are independent discrete log-concave with arbitrary expectation.
    
\end{abstract}

\vskip5mm
\noindent
{\bf Keywords:} Feige's conjecture, Sum of random variables, Small deviations, Log-concave.


\section{Introduction}

Motivated by the problem of estimating the average degree of a graph, Feige \cite{FEI} investigated the probabilistic quantity, $\mathbb{P}(X< \mathbb{E}[X]+ 1)$, where $X$ is the sum of $n$ independent non-negative random variables $X_1, X_2, \dots, X_n,$ with $\mathbb{E}[X_i] \leq 1$ for each $i$. Classical inequalities such as Markov's and Chebyshev's inequalities yield no useful information about this probability. Chebyshev's inequality does not play a role in this case since there is no assumption on the variance of $X$, while Markov's inequality implies $\mathbb{P}(X< \mathbb{E}[X]+ 1)\geq \frac{1}{\mathbb{E}[X]+1},$ which is essentially useless when $n$ is large. Using an approach based on a case analysis, Feige managed to prove that
\begin{equation}\label{Feige-bound}
        \mathbb{P}(X< \mathbb{E}[X] + 1 )\geq \frac{1}{13}.
\end{equation}
However, as Feige noted, one may take a collection of $n$ number of i.i.d. random variables such that for each $i, X_i = n+1$ with probability $\frac{1}{n+1}$ and $X_i=0,$ otherwise. Then, $$\mathbb{P}(X< \mathbb{E}[X] + 1 ) = \left(1- \frac{1}{n+1}\right)^n.$$  Based on this, it was conjectured that one could replace $1/13$ with $1/e.$ The improvement on inequality \eqref{Feige-bound} was first made by He, Zhang and Zhang (see \cite{HZZ}) by establishing the lower bound $1/8$. Later, Garnett (see \cite{GB}) improved the lower bound to $7/50$. The current best bound in this direction is $0.1798$ by Guo, He, Ling and Liu (see \cite{GHLL}).

Feige's inequality has many applications in computer science and combinatorics including the weighted max-cut problem (see \cite{HZZ}), approximating the average degree of a graph in sublinear-time (see \cite{FEI}, \cite{GRS}), and the connection of a  conjecture of Manickam, Miklós, and Singhi with matchings and fractional
covers of hypergraphs (see \cite{AHS}). We refer the reader to \cite{ABPY}, \cite{FK} and \cite{V} for more applications.

Our goal in this article is to prove that the conjectured lower bound holds when the collection $\{X_i\}_{i=1}^n$ is independent discrete log-concave. Recall that an integer-valued random variable $X$ is said to be log-concave if its probability mass function $p$ satisfies
\begin{equation*}
    p(k)^2 \geq p(k-1)\,p(k+1),
\end{equation*} for all $k \in \mathbb{Z}$ and $X$ has contiguous support. For example, Bernoulli, discrete uniform, binomial, negative binomial, geometric, hypergeometric and Poisson distributions are all log-concave. Many sequences in combinatorics are log-concave (or
conjectured to be log-concave), see, e.g., \cite{Mas}, \cite{SI}, \cite{S}, \cite{PB}, \cite{ALGV}. The surveys \cite{SII} and \cite{BR} provide more details about the notion of log-concavity in the context of combinatorics. Recent developments on discrete log-concavity in probabilistic setting include log-Sobolev type inequalities \cite{J}, discrete R\'{e}nyi entropy inequalities \cite{MT}, \cite{MMR}, \cite{MP}, concentration bounds and moments inequalities \cite{AMM}, \cite{MM2}. 

Our main result is as follows.

\begin{theorem}\label{feige}

Let $X$ be a discrete log-concave random variable. Then, \begin{equation}\label{feige-ineq}
    \mathbb{P}(X< \mathbb{E}[X] + 1) \geq e^{-1}.
\end{equation}

\end{theorem}

The lower bound $e^{-1}$ is sharp among the class of discrete log-concave random variables as can be seen by taking a random variable with probability mass function $p(k) = C n^{k/n}$, $k \in \{1, \dots, n\}$, with $n$ sufficiently large, where $C$ is the normalizing constant; the details are carried out in Section \ref{sharpness}.

Since the class of log-concave distributions is stable under reflection, Theorem \ref{feige} yields analogous bounds for upper tails. Namely, for any discrete log-concave random variable $X$,
$$ \P(X > \E[X] - 1) \geq e^{-1}. $$

Our main result is of independent interest, as it provides a sharp small deviation inequality on the class of log-concave distributions, a rich family that, in addition to containing the well known and studied distributions mentioned above, includes many interesting sequences for which explicit expressions are intractable. For instance, the intrinsic volumes of a convex body are log-concave by the Alexandrov-Fenchel inequality \cite{M91}, and hence by Theorem \ref{feige} we have the following immediate corollary. We refer to \cite{ALMT}, \cite{LMNPT} for further information on intrinsic volumes random variables and its importance in statistical learning.


\begin{corollary}

The central intrinsic volume of the intrinsic volume random variable associated with a convex body plus 1 is no smaller than its $1/e$ quantile.

\end{corollary}

Similarly, by the resolution of the strong Mason conjecture \cite{AHK} the number of independent sets of size $k$ in a matroid is log-concave. Therefore, Theorem \ref{feige} immediately implies the following corollary.

\begin{corollary}

The proportion of sub-forests of size smaller than the average forest size plus 1 is at least $1/e$.

\end{corollary}

Since the sum of independent discrete log-concave random variables is log-concave (see \cite{Fe}), Theorem \ref{feige} can be extended to $X = \sum_{i=1}^n X_i,$ where $X_i$'s are independent discrete log-concave. This proves a stronger version of Feige's conjecture as the optimal lower bound holds without any assumption on expectation and on the sign of the random variables. Note that this cannot be true for distributions that are not log-concave in general, this can be seen by taking, for example, a random variable $X$ such that $\P(X=0) = 1-p$, $\P(X=m) = p$, with $m$ sufficiently large, and $p$ sufficiently close to 1. Thus, in general, the constraint on the expectation is necessary.

Inequality \eqref{feige-ineq} has been established for binomial and the sum of independent Bernoulli random variables by Garnett \cite{GBT}. Theorem \ref{feige} extends Garnett's result to the whole class of log-concave probability sequences. The following stronger inequality has been proved for Poisson distribution by Teicher \cite{HT},
\begin{equation*}
    \mathbb{P}(X \leq \mathbb{E}[X])> e^{-1}.
\end{equation*}
However, Teicher's inequality does not hold for all log-concave random variables. This can be seen by taking a truncated Poisson distribution with the parameter equals to $5$, supported on $\{0,1,2\}$ so that $\P(X\leq \mathbb{E}[X])< e^{-1}.$

In the special case where $\mathbb{E}[X] \in \mathbb{Z}$, Theorem \ref{feige} yields the following result.

\begin{cor}\label{GreenMoh}

Let $X$ be a discrete log-concave random variable. If $\mathbb{E}[X] \in \mathbb{Z},$ then
\begin{equation}\label{integer}
    \mathbb{P}(X \leq \mathbb{E}[X]) \geq e^{-1}.
\end{equation}

\end{cor}

For specific distributions, the lower bound in corollary \ref{GreenMoh} can be improved to $1/2$ (when $\E[X]$ is an integer), as shown for the sum of independent Bernoulli (see \cite{AS}), Poisson (see \cite{HT}), and can easily be verified for the discrete uniform distribution. However, this is not true for the whole class of log-concave distributions, as can be seen by taking a random variable supported on $\{1, \dots, 8\}$ with probability mass function $p(k) = Cp^k$, $k \in \{1, \dots, 8\}$, where $p$ is chosen so that $\E[X] = 6$ and $C$ is the normalizing constant. For this distribution, $\P(X \leq \E[X]) < \frac{1}{2}$.

In the continuous setting, a result of Gr\"unbaum \cite{Gr} (see also \cite[Lemma 2.4]{MPa}) states that if $X$ is a continuous log-concave random variable then the following stronger inequality holds,
\begin{equation}\label{grunbaum}
\P(X \leq \E[X]) \geq \frac{1}{e}.
\end{equation}
Recall that a real-valued random variable is log-concave if it has a probability density function $f$ (with respect to Lebesgue measure) satisfying
$$ f((1-\lambda) x + \lambda y) \geq f(x)^{1-\lambda} f(y)^{\lambda}, $$
for all $x,y \in \R$ and $\lambda \in [0,1]$. Examples include Gaussian, exponential, and uniform over an interval. See, e.g., \cite{KM}, \cite{SW}, \cite{C}, for properties and applications of continuous log-concave distributions. Since the sum of independent continuous log-concave random variables is log-concave (see \cite{Pr}), inequality \eqref{grunbaum} applies to $X = \sum_{i=1}^n X_i,$ where $X_i$'s are independent continuous log-concave, in particular Feige's conjecture holds for this class of random variables.

Let us note that Theorem \ref{feige} implies inequality \eqref{grunbaum} for continuous log-concave random variables. Indeed, suppose that $X$ is log concave with support $[0,1]$ and density function $f$. For integer $n \geq 1$, define $X_n$ to be the discrete log-concave random variable on $\{1, \dots, n\}$ with probability mass function
$$ \P(X_n = k) = \frac{f(k/n)}{\sum_{j=1}^n f(j/n)}. $$
One can check that $X_n/n$ converges to $X$ in distribution, and together with Theorem \ref{feige} it follows that
$$ \P(X \leq \E[X]) = \lim_{n \to +\infty} \P \left( \frac{X_n}{n} \leq \frac{\E[X_n]}{n} + \frac{1}{n} \right) = \lim_{n \to +\infty} \P(X_n \leq \E[X_n] + 1) \geq \frac{1}{e}. $$
The proof of the result for general $X$ follows by approximation on compact sets and rescaling.

The article is organized as follows. In section \ref{proof}, we prove Theorem \ref{feige}. In section \ref{sharpness}, we comment on the sharpness of our main result.

\section{Proof of Theorem \ref{feige}}\label{proof}

The key idea of the proof of Theorem \ref{feige} is to reduce the problem to truncated geometric distributions. This is due to the identification of the extreme points of the convex hull of a subset of discrete log-concave probability distributions satisfying a linear constraint and the standard use of Krein-Milman's theorem, developed by the third and fourth named authors in \cite{MM} (see also \cite{NS}). This method can be seen as a discrete analogue of a localization technique due to Kannan, Lov\'asz and Simonovits \cite{KLS} in the form of Fradelizi and Gu\'edon \cite{FG}. For completeness, we recall the main argument from \cite{MM}.

Let $M, N \in \Z$. Denote by $\mathcal{P}(\{M, \dots, N\})$ the set of all discrete probability measures supported on $\{M, \dots, N\}$. Let $h \colon \{M, \dots, N\} \to \mathbb{R}$ be an arbitrary function. Consider $\mathcal{P}_h(\{M, \dots, N\})$ the set of all log-concave distributions $\P_X$ in $\mathcal{P}(\{M, \dots, N\})$ satisfying the constraint $\E[h(X)] \geq 0$, that is,
$$ \mathcal{P}_h(\{M, \dots, N\}) = \{ \P_X \in \mathcal{P}(\{M, \dots, N\}) : X \mbox{ log-concave, } \, \E[h(X)] \geq 0 \}. $$
The following theorem describes the shape of the extreme points of $\Conv(\mathcal{P}_h(\{M, \dots, N\}))$.

\begin{theorem}[\cite{MM}]\label{extreme}

 If $\P_X \in \Conv(\mathcal{P}_h(\{M, \dots, N\}))$ is an extreme point, then its probability mass function $p$ satisfies
    \begin{equation}\label{extremizers}
        p(k) = C p^k 1_{\{m, \dots, n\}}(k),
    \end{equation}
    for some $C, p >0$, $m,n \in \{M, \dots, N\}$.

\end{theorem}

One can therefore deduce by the (finite-dimensional version of the) Krein-Milman theorem that the supremum of any convex functional over the set $\mathcal{P}_h(\{M, \dots, N\})$ is attained at probability distributions of the form \eqref{extremizers} (see \cite{MM}). 

The convex (linear) functional that will be considered in the proof of Theorem \ref{feige} is of the form $\Phi \colon \P_X \mapsto \P_X(A)$ for a fixed Borel set $A \subset \R$.

The next lemma shows that in order to prove Theorem \ref{feige} it is enough to consider compactly supported log-concave random variables.

\begin{lemma}\label{approx}

If the inequality
\begin{equation}\label{to-prove}
\P(X < \E[X] + 1) \geq \frac{1}{e}
\end{equation}
holds for all compactly supported log-concave random variable $X$, then the inequality holds for all log-concave random variables on $\Z$.

\end{lemma}

\begin{proof}
By assumption, inequality \eqref{to-prove} holds for all compactly supported log-concave random variables. It remains to prove the inequality for log-concave random variables either with full support on $\Z$, lower bounded, or upper-bounded.

Assume first that $X$ is a log-concave random variable with support $\{1,2,\dots\}$ and probability mass function $p$. For $n \geq 1$, define a log-concave random variable $X_n$ with probability mass function
$$ p_n(k) = \frac{p(k)}{\sum_{j=1}^n p(j)}, \quad k \in \{1, \dots, n\}. $$
Note that for all $n \geq 1$,
\begin{eqnarray*}
\E[X_n] \leq \E[X] \Longleftrightarrow  \sum_{k=1}^n k p(k) \sum_{k \geq n+1} p(k) \leq \sum_{k=1}^n  p(k)  \sum_{k \geq n+1} k p(k),
\end{eqnarray*}
which is true since
$$ \sum_{k=1}^n k p(k) \sum_{k \geq n+1} p(k) \leq (n+1) \sum_{k=1}^n p(k) \sum_{k \geq n+1} p(k) \leq \sum_{k=1}^n  p(k)  \sum_{k \geq n+1} k p(k). $$
Therefore, for all $n \geq 1$,
$$ \P(X_n < \E[X] + 1) \geq \P(X_n < \E[X_n] + 1) \geq \frac{1}{e}, $$
where the last inequality comes from the assumption, since $X_{n}$ is compactly supported. We conclude by letting $n$ goes to $+\infty$ as
$$ \P(X_n < \E[X] + 1) \underset{n \to +\infty}{\longrightarrow} \P(X < \E[X]+1). $$
Since inequality \eqref{to-prove} is translation invariant, we deduce that the inequality holds for all lower bounded log-concave random variables.

Assume now that $X$ is a log-concave random variable with support $\{\dots, -2, -1\}$ and probability mass function $p$. Define, for $i > 1/p(-1)$, a log-concave random variable $X_i$ with probability mass function
$$ p_i(k) = C_i p(k), \quad k \in \{\dots, -3, -2\}, \quad \mbox{and} \quad p_i(-1) = C_i \left( p(-1) - \frac{1}{i} \right), $$
where $C_i = (1-\frac{1}{i})^{-1}$ is the normalizing constant. Note that for all $i > 1$,
$$ \E[X_i] = \frac{i}{i-1} \left( \sum_{k \leq -2} k p(k) - p(-1) + \frac{1}{i} \right) = \frac{i}{i-1} \left( \E[X] + \frac{1}{i} \right) < \E[X], $$
where the last inequality follows from $\E[X] < -1$. For $i > 1/p(-1)$ and $m \geq 1$, define a compactly supported log-concave random variable $X_{i,m}$ with probability mass function
$$ p_{i,m}(k) = \frac{p_i(k)}{\sum_{j =-m}^{-1} p_i(j)}, \quad k \in \{-m, \dots, -1\}. $$
Note that $\E[X_{i,m}] \to \E[X_i]$ as $m \to +\infty$, therefore, since $\E[X_i] < \E[X]$, there exists $m_0 \in \N$ such that for all $m \geq m_0$, $\E[X_{i,m}] \leq \E[X]$. Hence, as above,
for $m \geq m_0$,
$$ \P(X_{i,m} < \E[X] + 1) \geq \frac{1}{e}, $$
and we conclude by letting $i,m \to +\infty$. Again by translation invariance, the inequality then holds for all upper bounded log-concave random variables.

Finally, assume that $X$ is a log-concave random variable fully supported on $\Z$. Denote by $p$ the probability mass function of $X$. The approximation $X_n$, $n \geq 1$, with probability mass function
$$ p_n(k) = \frac{p(k)}{\sum_{j \leq n} p(j)}, \quad k \in\{\dots, -1,0,1, \dots, n\}, $$
yields the result by a similar argument as above, since for all $n \geq 1$, $X_n$ is upper bounded and $\E[X_n] \leq \E[X]$.
\end{proof}

\begin{remark}

Let us note that a similar construction as in Lemma \ref{approx}, combining truncation and scaling of probability mass function, may be used to show that given a random variable $X$ on $\Z$ with finite absolute moment, one may construct a sequence of compactly supported random variables $\{X_n\}$ such that $\E[X_n] = \E[X]$, $n \geq 1$, and $\{X_n\}$ converges to $X$ in total variation. 


\end{remark}

\begin{proof}[Proof of Theorem \ref{feige}]
First, by Lemma \ref{approx}, it is enough to prove inequality \eqref{feige-ineq} for compactly supported log-concave random variables, say on $\{M, \dots, N\}$, for arbitrary $M \leq N$. Next, according to Theorem \ref{extreme}, it is enough to consider log-affine distributions supported in $\{M, \dots, N\}$, that is, distributions of the form
$$ p(k) = C p^k, \quad k \in \{m \dots, n\}, $$
for all $p>0$ and $M \leq m \leq n \leq N$. Finally, if $X$ is log-affine on $\{m, \dots, n\}$, then $\widetilde{X} = X - m + 1$ is log-affine supported on $\{1, \dots, \widetilde{n}\}$, where $\widetilde{n} = n-m+1$, and
$$ \P(X \geq \E[X] + 1) = \P(\widetilde{X} \geq \E[\widetilde{X}] + 1). $$
Therefore, it is enough to prove the desired inequality for log-affine distributions supported on $\{1, \dots, n\}$, for arbitrary $n \geq 1$. In the following, we consider $X$ with probability mass function of the form
$$ p(k) = Cp^k, \quad k \in \{1, \dots, n\}, $$
for $n \geq 1$ and $p >0$, and we will establish inequality \eqref{feige-ineq} for such random variables. If $p = 1$, then
$$ \P(X < \E[X] + 1) \geq \frac{1}{2} \geq \frac{1}{e}. $$
When $p \neq 1$, we have
$$ C = \frac{1-p}{p(1-p^n)}. $$
Therefore,
\begin{eqnarray*} \E[X] = Cp \left[ \sum_{k=0}^n p^k \right]' = \frac{np^{n+1} - (n+1)p^{n} + 1}{(1-p)(1-p^n)}  = \frac{1}{1-p} - \frac{n p^n}{1-p^n}.
\end{eqnarray*}
On the other hand,
$$ \P(X < \E[X] + 1) = \sum_{k = 1}^{\lceil \E[X] \rceil} Cp^k = \frac{1 - p^{\lceil \E[X] \rceil}}{1 - p^n}. $$
Thus we are left to prove that
$$ \frac{1 - p^{\lceil \E[X] \rceil}}{1 - p^n} \geq \frac{1}{e}. $$
Since
$$ \frac{1 - p^{\lceil \E[X] \rceil}}{1 - p^n} \geq \frac{1 - p^{\E[X]}}{1 - p^n}, $$
it suffices to show that for all positive integer $n$ and every positive real number $p \neq 1$,
\begin{equation*}
    \frac{1 - p^{\frac{1}{1-p} - \frac{n p^n}{1-p^n}}}{1 - p^n} \geq \frac{1}{e}.
\end{equation*}
Substituting $x=p^n$, the above inequality is equivalent to
$$ \frac{1 - x^{\frac{1}{n(1-x^{1/n})} - \frac{x}{1-x}}}{1 - x} \geq \frac{1}{e}. $$
Note that the left hand side is non-increasing in $n$ since
$$ \frac{d}{dt} \left[ t(1-x^{1/t}) \right] = 1-x^{1/t} + x^{1/t} \log(x^{1/t}) \geq 0, $$
where the substitution $y = x^{1/t}$ makes the inequality obvious. Taking the limit in $n$, it suffices to prove that for $x \neq 1$,
\begin{equation*}\label{reduc}
    g(x) := \frac{1 - x^{-\frac{1}{\log(x)} - \frac{x}{1-x}}}{1 - x} \geq \frac{1}{e}.
\end{equation*}
Note that $\lim_{x \to +\infty} g(x) = \frac{1}{e}$, so the result would follow from
$$ 0 \geq g'(x) = \frac{e(x-1) - x^{\frac{x}{x-1}} \log(x)}{e(x-1)^3}. $$
To prove that $g'(x) \leq 0$, it suffices to prove that
$$ h(x) := \frac{x^{\frac{x}{x-1}} \log(x)}{e(x-1)} \geq 1. $$
Computing,
$$ h'(x) = \frac{x^{\frac{1}{x-1}} \left( (x-1)^2 - x\log(x)^2 \right)}{e(x-1)^3}, $$
the result will follow from $k(x) := (x-1)^2 - x\log(x)^2 \geq 0$ as this will give $h'(x) \leq 0$ for $x<1$ and $h'(x) \geq 0$ for $x>1$, so that $h(x) \geq h(1) = 1$. Since $k$ is convex, as $k''(x) = \frac{2}{x}(x-1-\log(x)) \geq 0$, and $k'(1) = 0$, we deduce that $k$ takes its minimum value $0$ when $x=1$, completing the proof.
\end{proof}

\section{Sharpness of Theorem \ref{feige}}\label{sharpness}

This section comments on the sharpness of Theorem \ref{feige}.

\begin{enumerate}
    \item For a discrete log-concave random variable $X$, one trivially has the extension to $t \geq 1$,
$$ \P(X \geq \E[X] + t) \leq 1 - \frac{1}{e}. $$
Let us see that these inequalities are sharp for any $t \geq 1$. Fix $n \geq 2$ and let $X$ be log-affine on $\{1, \dots, n\}$ with parameter $p = n^{\frac{1}{n}}$. Following the computation of the proof of Theorem \ref{feige}, one has
$$ \frac{n - n^{\frac{\E[X] + t - 1}{n}}}{n-1} \geq \P(X \geq \E[X] + t) \geq \frac{n - n^{\frac{\E[X] + t}{n}}}{n-1}. $$
Thus to prove that $1 - \frac{1}{e}$ is optimal among log-concave distributions it is enough to prove that for $\lambda \geq 0$,
$$ \lim_{n \to +\infty} \frac{n - n^{\frac{\E[X] + \lambda}{n}}}{n-1} = 1 - \frac{1}{e}. $$
The result follows since
$$ \frac{n}{n-1} - \frac{1}{n^{\frac{n+1}{n}} - n} = 1 - \frac{1}{\log(n)} + O \left( \frac{1}{n} \right), $$
therefore, using $\E[X] = \frac{n^2}{n-1} - \frac{1}{n^{1/n} - 1}$, we have
$$ \frac{n - n^{\frac{\E[X] + \lambda}{n}}}{n-1} = \frac{n}{n-1} \left( 1 - n^{\frac{\lambda}{n}} n^{-\frac{1}{\log(n)} + O \left( \frac{1}{n} \right)}  \right) \xrightarrow[n \to +\infty]{} 1 - \frac{1}{e}. $$

    \item Theorem \ref{feige} is sharp also within the subclass of ultra log-concave random variables. First, recall that a random variable $X$ is said to be ultra log-concave (ULC) if its probability mass function $p$ satisfies
    \begin{equation*}
    p(k)^2 \geq \frac{k+1}{k} p(k-1)\,p(k+1),
    \end{equation*} for all $k \in \mathbb{N} = \{1, 2, \dots\}$ and $X$ has contiguous support. Let
    $$ C_{LC} = \sup \{C : X \mbox{ log-concave } \Longrightarrow \P(X < \E[X] + 1) \geq C \}, $$
    and
    $$ C_{ULC} = \sup \{ C : X \mbox{ ultra log-concave } \Longrightarrow \P(X < \E[X] + 1) \geq C \}. $$
    We will show that
    $$ C_{ULC} = C_{LC}. $$
    Clearly, if $C$ gives a lower bound on $\P(X < \E[X] + 1)$ for every $X$ log-concave, it gives a lower bound if $X$ is restricted to the class of ultra log-concave variables, hence
    $$ C_{LC} \leq C_{ULC}. $$
    To argue in the reverse direction, 
    let $\widetilde{X}$ be the truncation of a Poisson random variable to the interval $\{m, \dots, m + n\}$ and with parameter $\lambda m$. Since $\widetilde{X}$ is the truncation of a ULC distribution, $\P( \widetilde{X} < \E[ \widetilde{X}] + 1) \geq C_{ULC}$. However, by translation invariance of the inequality, $X_m = \widetilde{X} - m$ also satisfies
    $$ \P(X_m < \E[X_m] + 1) \geq C_{ULC}. $$
    The probability mass function of $X_m$ can be written as
    $$ p_m(k) = K_m \frac{(\lambda m )^{m+k}}{(m+k)!} 1_{\{0, \dots,n\}}(k), $$
    where $K_m$ is the normalizing constant. Taking the limit in $m \to \infty$ of the ratios for $k \in \{0, \dots, n\}$,
    $$ \frac{p_m(k+1)}{p_m(k)} = \lambda \frac{m}{m+k+1} \longrightarrow \lambda. $$
    Thus as $m \to \infty$, $X_m$ converges in distribution to a log-affine distribution with parameter $\lambda$ on $\{0, \dots, n\}$. Thus for any compactly supported log-affine $X$,
    $$ \P(X \leq \E[X] + 1) \geq C_{ULC}, $$
    and hence by localization (see proof of Theorem \ref{feige})
    $$ \P(X \leq \E[X] + 1) \geq C_{ULC} $$
    holds for all log-concave distributions $X$, and we conclude that $C_{ULC} \leq C_{LC}$ by the first remark above on the sharpness of $\P(X < \E[X] + t)$, for all $t \geq 1$.
\end{enumerate}

\end{document}